\documentclass[reqno]{amsart}

\usepackage{enumerate}
\usepackage{ifpdf}
\usepackage{amsmath}
\usepackage{amsfonts}
\usepackage{amssymb}
\usepackage{amsthm}
\usepackage{stmaryrd}
\usepackage{setspace}
\usepackage{amsrefs}
\usepackage{paralist}
\usepackage{faktor}
\newcommand{\ignore}[1]{}

\renewcommand{\Im}{\operatorname{Im}}

\newcommand{\C}{{\mathbb{C}}}
\newcommand{\R}{{\mathbb{R}}}

\newcommand{\N}{{\mathbb{N}}}

\newcommand{\sA}{{\mathcal{A}}}

\newcommand{\sO}{{\mathcal{O}}}

\newcommand{\sX}{{\mathcal{X}}}

\newcommand{\vnorm}[1]{\left\|  #1 \right\|}

\newcommand{\Cn}{\mathbb{C}^n}

\newcommand{\dopt}[2]{\frac{\partial #1}{\partial #2}}

\newcommand{\nequiv}{{\equiv \!\!\!\!\!\!  / \,\,}}

\newcommand{\holmaps}{\mathcal{H}}

\newcommand{\fps}[1]{\C\llbracket #1 \rrbracket}

\newcommand{\cps}[1]{\C\{#1\}}

\newcommand{\diffable}[1]{\mathcal{C}^{#1}}
\newcommand{\idealsheaf}{\mathcal{I}}

\newcommand{\inp}[1]{\langle #1 \rangle}
\newcommand{\averaging}{\mathcal{A}}
\newcommand{\restricted}{\mathcal{R}}
\newcommand{\hompoly}{\mathcal{P}}

\newtheorem{thm}{Theorem}[section]

\newtheorem{prop}[thm]{Proposition}

\newtheorem{lemma}[thm]{Lemma}

\theoremstyle{definition}
\newtheorem{defn}[thm]{Definition}
\newtheorem{example}[thm]{Example}

\theoremstyle{remark}

\author{Bernhard Lamel}
\author{Ji\v{r}\'i Lebl}
\date{\today}

\title{Segre nondegenerate totally real subvarieties}


\begin{document}


\begin{abstract}
We study an irreducible real-analytic germ of an $n$-dimensional variety
in $n$ dimensional complex space.  Assuming that the variety is
Segre nondegenerate we define an averaging operator that generalizes
the Moser--Webster involution.  This operator can be thought of as being the
CR structure of the singularity, and using this operator we study the set of
functions that are restrictions of holomorphic functions.  We give a
condition on the flattening of the singularity, that is realizing the
singularity as a codimention one subvariety of a nonsingular Levi-flat
hypersurface.
\end{abstract}

\maketitle



\section{Introduction}

A natural question in complex analysis is the following:

\medskip

\emph{Given a set $X
\subset \C^n$, characterize those functions $f \colon X \to \C$ that are
restrictions of holomorphic functions defined in a neighborhood of $X$.}

\medskip

When $X$ is a real-analytic
CR submanifold, then the answer is well-understood, it is the set of
real-analytic CR functions, that is functions that satisfy the
Cauchy--Riemann equations restricted to $X$.  If $X$ is a generic
$n$-dimensional submanifold, such as $X=\R^n \subset \C^n$,
then the CR structure of $X$ is trivial, and every
real-analytic function is the restriction of a holomorphic function.

If $X$ is singular, then the answer is much more difficult.  The case that
we are interested in is the local question near a singular point
of a real-analytic subvariety of dimension $n$.  At a nondegenerate
singular point we will be define a finite alternative to the Cauchy--Riemann
equations.  The CR equations normally say that a holomorphic function is
constant along a certain (complex) direction.  In our setting, we will
replace this complex direction with a finite set of points.
The setting applies also to CR singular manifolds, and given a Bishop
surface, that is a 2 dimensional real submanifold of $\C^2$ with a
nondegenerate complex tangent, see \cite{MR200476}.
This finite set of points are precisely the two from the
Moser--Webster involution, see \cite{MR709143}.
In $\C^2$, the CR singular case, focused mainly on normal forms,
was studied further by
Moser~\cite{MR802502},
Kenig--Webster~\cites{MR664323},
Harris~\cite{MR773068},
Gong~\cite{MR1303227},
Huang--Krantz~\cite{MR1328757}, 
Huang--Yin~\cite{MR2501295},
and others.
Harris~\cite{MR477120} studied the restriction question on
a CR singular submanifold in
terms of vector fields defined on $M$.
Lebl--Noell--Ravisankar~\cite{MR3663330} proved that functions
satisfying a moment condition on the elliptic Bishop surface
are restrictions of holomorphic functions.
CR singular submanifolds of dimension $n$ have
similarly been studied by Webster~\cite{MR800003},
Kenig--Webster~\cite{MR764946},
Huang~\cite{MR1603854},
Coffman~\cite{MR2650710}, 
Ahern--Gong~\cite{MR2518102},
Gong--Stolovich~\cites{MR3570294,MR3948229}
and others.
See also the survey by Huang~\cite{MR3647129}.

The Segre variety of an
$n$-dimensional real subvariety of $\C^n$ is generically a finite set
of points.
As one can average over a finite set of points, we obtain an operator
from the real power series $\cps{z,\bar z}$ to the holomorphic power
series $\cps{z}$.  The operator is given explicitly, and can be computed
for a given subvariety up to any given order.  The operator reproduces the
holomorphic functions, and therefore can be used to answer the motivating
question: What are the restrictions of holomorphic functions?

The operator can also be used to attempt to find normal forms for $X$.
As $\R^n$ is given by the vanishing of the imaginary part of $n$ independent
holomorphic functions, one can similarly ask to find at least one such
function for a singular $X$.  Such a function we call a ``flattening'',
as it gives a Levi-flat hypersurface that contains $X$.
The flattening question is equivalent to finding
a holomorphic function $f$ such that $\bar{f} = f$ on $X$, or in other words
when the averaging operator is applied to $\bar{f}$ it simply yields $f$.
As the operator works formally, one can use it to explicitly find
obstructions to flattening.

The flattening question only requires the restriction of the averaging
operator to $\cps{\bar{z}}$.  It turns out that this restriction uniquely
describes the germ of $X$ at a point.  That is, finding local
normal forms for $X$ is equivalent to finding the normal forms for
the restricted operator.

Let us state our main results more precisely.
Let $(X,0) \subset (\C^n,0)$ be an (irreducible) $n$-dimensional
germ of a real-analytic subvariety at the origin in $\Cn$. We denote
by $\idealsheaf_0 (X) \subset \cps{z,\bar z}$ the ideal of germs 
of real-analytic functions vanishing on $X$. We say that  $(X,0)$ is
\emph{Segre nondegenerate} 
if the ideal 
$ \mathcal{J} = \left\{ \varrho(z,0) \colon \varrho \in \idealsheaf_0 (X) \right\} \subset \cps{z} $
is an ideal of definition, i.e. if there exists a $k\in \N$ such 
that the maximal ideal $\mathfrak{m}\subset\cps{z}$ 
satisfies $\mathfrak{m}^k \subset \mathcal{J}$, or equivalently,
if its vanishing locus satisfies $V(\mathcal{J}) = \left\{ 0 \right\}$.

Let $\sO_0$ be the ring of germs of holomorphic functions.  We say a
germ $(f,0)$ of a function is real-analytic on $(X,0)$ if there is a
germ of a real-analytic function
defined in $(\C^n,0)$ whose restriction to $(X,0)$ is $(f,0)$.  In this case
we will often identify $(f,0)$ with its extension.  We say $(f,0)$ is a
restriction of a holomorphic function if there exists an extension
such that $(f,0) \in \sO_0$.

To make notation easier, we will sometimes drop the $(\cdot,0)$ notation when not
absolutely necessary for clarity.  Given a small enough neighborhood,
a real-analytic function has a unique representative and we will generally
identify the germ with one of its representatives.

Let $(X,0) \subset \C^n$ be an irreducible germ of a Segre nondegenerate
$n$-dimensional real-analytic subvariety of multiplicity $k$.  That is,
the Segre variety generically has $k$ points.
Given a real-analytic function $f(z,\bar{z})$ we let
$\averaging f$ be the average $\frac{1}{k} \sum_{j} f(z,\xi_j)$
over the points $\xi_1,\ldots,\xi_k$ of the Segre variety at $z$.

Our first main result is that $f$ is a restriction of a holomorphic
function if and only if
\begin{equation*}
\averaging (f^\ell) = {(\averaging f)}^\ell \text{ for all } \ell = 1,\ldots,k.
\end{equation*}

Let us call $\restricted$ the restriction of
$\averaging$ to the antiholomorphic functions.  It turns out that
$\restricted$ contains all the information about $(X,0)$ and $\averaging$,
that is, finding a normal form for $\restricted$ is equivalent to
finding a normal form for $(X,0)$.

A natural question about the normal form of $(X,0)$ is the so-called
flattening.  That is, does there exist a holomorphic function
that is real-valued on $(X,0)$.  We prove that
a holomorphic function $f$ is real-valued
on $X$ if and only if
\begin{equation*}
\restricted (\bar{f}^\ell) = f^\ell \text{ for all } \ell = 1,\ldots,k.
\end{equation*}

Let us outline the structure of this paper.
In \S\ref{sec:prelim}, we explain the notation and the setup
of the problem including complexification and the Segre varieties.
In \S\ref{sec:obstr}, we study the obstructions for a function to be
the restriction of a holomorphic function.
In \S\ref{sec:averaging}, we define the averaging operator.
In \S\ref{sec:flattening}, we discuss the restricted averaging operator,
and its application to flattening and show that it contains all
the necessary information to define $X$.
In \S\ref{sec:examplesofflat}, we work out the flattening in some examples.

\section{Preliminaries} \label{sec:prelim}

We start with some properties of the extrinsic complexification 
of a Segre-nondegenerate germ: we recall that 
$(\mathcal{X},0) \subset (\C^{2n}_{z,\xi},0)$ is a complexification 
of $(X,0)$ if $\mathcal{X} \cap \left\{ z = \overline{\xi} \right\}$ as 
germs at the origin. We denote the projections onto 
the first and the second factor by 
$\pi_1 (z,\xi) = z$ and $\pi_2 (z,\xi) = \xi$, respectively. 

\begin{prop}
Let $(X,0) \subset (\C^n_z,0)$ be a germ of a Segre nondegenerate
irreducible $n$-dimensional subvariety at the origin.
Then there exist polydiscs $\Delta_z \subset \C^n$ and $\Delta_\xi \subset
\C^n$, centered
at the origin,
and an $n$-dimensional closed complex subvariety
$\sX \subset \Delta_z \times \Delta_\xi$ which is irreducible
both globally and at the origin, such that $(\sX,0)$ is
the complexification of $(X,0)$, and
an integer $k$, such that $\pi_1|_{\mathcal{X}}$ is finite, 
and $(\pi_1|_{\mathcal{X}})^{-1} (z)$ consists of $k$ 
points counting multiplicity for 
$z \in \Delta_z$, and furthermore $(\pi_1|_{\mathcal{X}})^{-1}(0) = \{ 0 \}$.

The same statement holds with a possibly different pair 
of polydiscs $\tilde\Delta_z \subset \C^n$ and $\tilde\Delta_\xi \subset
\C^n$ and the projection onto the second coordinate $\pi_2$.
\end{prop}

\begin{defn}
If $(X,0) \subset (\C^n,0)$ is a germ at 0 of a Segre nondegenerate
$n$-dimensional subvariety, then we call
the polydiscs $\Delta_z \times \Delta_\xi$ \emph{$\pi_1$-good} for $(X,0)$,
if they are small enough as above,
and we will call the closed subvariety $\sX \subset \Delta_z
\times \Delta_\xi$ the \emph{corresponding complexification}; the polydiscs $\tilde\Delta_z \times \tilde\Delta_\xi $ will be said to be {\em $\pi_2$-good}, and the corresponding complexification
is defined likewise. We will 
call $k$ the {\em Segre multiplicity} of $(X,0)$.
\end{defn}

The germ of $\mathcal{X}$ at the origin 
is well defined, and for small neighbourhoods
the corresponding complexification is 
just a representative of that germ in 
that neighbourhood. We will therefore use
 $\mathcal{X}$ as a notation 
 for the complexification in any (small
 enough) neighbourhood.

\begin{proof} We only the prove the corresponding 
statement for the $\pi_1$-good polydiscs; the $\pi_2$-good polydiscs 
are done analogously. 
First we find a small enough neighborhood of the origin
and a complexification that is irreducible at the origin.
That follows by simply taking the smallest germ
of a complex subvariety that contains the germ of the
set
\begin{equation*}
\{ (z,\xi) : \bar{z} = \xi , z \in X \} 
\end{equation*}
at the origin, for some representative $X$ of $(X,0)$.
If $\sX$ were not irreducible, it would imply
$(X,0)$ is also reducible by restricting the
components of $\sX$ to the diagonal $\bar{z} = \xi$.
Making the neighborhood $\Delta_z \times \Delta_\xi$
small enough we can ensure that the
complexification is irreducible for any other smaller polydisc
neighborhood.

The subvariety $X$ is $n$-real-dimensional,
and hence $\sX$ is $n$-complex-dimensional.
Because $X$ is Segre nondegenerate, then the $n$-dimensional
subspace $\{ 0 \} \times \C^n \subset \C^n \times \C^n$ intersects
$\sX$ at an isolated point at the origin.  In particular,
this means that locally near the origin, $\sX$ is a multigraph
of a $k$-valued holomorphic mapping (see \cite{MR0387634}).
We can now 
choose $\Delta_z$ small enough that $(\pi_1|_{\mathcal{X}})^{-1}$ has generically $k$
preimages in $\Delta_\xi$ (exactly $k$ counting multiplicity),
and furthermore that $(\pi_1|_{\mathcal{X}})^{-1}(0)$ is the origin alone.
\end{proof}

In terms of the ideal $\idealsheaf_0 (X) \subset \cps{z,\bar z}$ 
of $(X,0)$, i.e. the set of 
germs of real-analytic functions at $0$ vanishing on $X$, 
we have that $\idealsheaf_0 (X) $ is 
a {\em real} ideal. That is, $\iota (\idealsheaf_0 (X) ) \subset \idealsheaf_0  (X)$, where 
$(\iota \varrho )(z, \bar z) = \bar \varrho( \bar z, z)$. One can check that the 
ideal of any complexification $\mathcal{X}$ (in either a $\pi_1$- or a $\pi_2$-good polydisc)  at the origin is given by 
$\idealsheaf_0 ( \mathcal{X}) = \left\{ \varrho(z,\xi) \colon
\varrho(z,\bar z) \in \idealsheaf_0 (X)  \right\}$. Now, 
because $\idealsheaf_0 (X) $ is real, we have that the involution
 $\iota\colon \cps{z,\xi} \to \cps{z,\xi} $ defined by 
 $\iota(\varrho) (z,\xi) = \bar \varrho(\xi, z)$ leaves 
 $\idealsheaf_0 (\mathcal{X})$ invariant. For a set $A\subset \Cn$ we 
 are going to denote the by $A^*$ the 
 set of complex conjugates of elements of $A$. The preceding algebraic fact has the following geometric
 interpretation: 

\begin{prop}\label{prop:reflection} Let $(X,0) \subset (\C^n_z,0)$ be a germ of a Segre nondegenerate
irreducible $n$-dimensional subvariety at the origin, $\mathcal{X}$ a complexification 
of $X$. Then for a small enough neighbourhood $\Delta$ of the origin,
the map $S \colon \Delta \times  \Delta^* \to \Delta \times  \Delta^*$ defined by $S(z,\xi) = (\bar \xi, \bar z)$ leaves $\mathcal{X}\cap(\Delta \times  \Delta^*) $ invariant.  
\end{prop}

\begin{proof}
We can choose real generators of $\idealsheaf_0 (\mathcal{X})$, i.e. 
germs $\varrho_1,\dots ,\varrho_p \in \cps{z,\xi}$ satisfying 
$\varrho_j (z, \xi) = \bar \varrho_j (\xi, z)$ for $j = 1, \dots ,p$.
Now, for $z$ and $\xi$ sufficiently close to $0$, 
$(z,\xi) \in \mathcal{X}$ if and only if $\varrho_j (z,\xi) = 0$
for all $j$, which in turn is equivalent to $\bar \varrho_j (\bar z, \bar \xi)= \varrho_j (\bar \xi, \bar z) =0$ for $j=1, \dots, p$, i.e. $(\bar \xi, \bar z) \in \mathcal{X}$. 
\end{proof}

Proposition~\ref{prop:reflection} allows us to identify representatives
of $(X,0)$ with the diagonal $z= \bar \xi$ in  $\mathcal{X}\cap (\Delta \times \Delta^*)$ for 
small polydiscs $\Delta$. Whenever we need to refer to a representative 
of the germ $(X,0)$, we will choose one of the form constructed in this proposition.

%
%


\begin{prop}\label{p:genericallytotallyreal}
Let $(X,0)$ is an irreducible $n$-dimensional germ of a real subvariety at the
origin.  If $(X,0)$ is Segre nondegenerate, 
then at a generic 
dimension-$n$ regular point of
a small enough representative, $X$ is
a maximally totally real submanifold.  
Furthermore, the germ $(X,0)$ is not
contained in any germ of a proper 
complex analytic subvariety at the origin.
\end{prop}

\begin{proof}
Let $\Delta_z \times \Delta_\xi$ be good for $(X,0)$ and let
$\sX$ be the corresponding complexification.
The projection $(\pi_1|_\mathcal{X})$ is generically $k$-to-1.  
If $(X,0) \subset (Y,0) $ for a germ of a 
proper complex analytic subvariety  
then 
$(\pi_1|_\mathcal{X})^{-1}(z)$ would be empty for $z \notin Y$.  
Hence, $(X,0)$ is not contained in any
proper complex subvariety.


Since the discriminant set of $\pi_1|_\mathcal{X}$ is a complex subvariety
in $\Delta_z$, at a 
 generic point $z_0 \in X$ we  have that 
$(z_0,\bar{z}_0)$ is a regular point of $\sX$,
and hence $z_0$ is a regular point of $X$.
Then locally near $(z_0,\bar{z}_0)$, $\sX$ can be written as
 a graph $\xi = g(z)$, and hence
$X$ near $z_0$ is given by the (vector) equation $\bar{z} = g(z)$.  In other
words $X$ near $z_0$ is a maximally totally real submanifold.
\end{proof}

\subsection{Symmetric functions and standard defining equations} 
\label{sub:symmetric_functions_and_standard_defining_equations}
We will now recall some standard facts about analytic varieties from 
Whitney's book \cite{MR0387634} in the setting we need. 

If $(X,0)$ is a Segre nondegenerate germ of multiplicity $k$, 
$\tilde\Delta_z \times \tilde\Delta_\xi$ 
is $\pi_2$-good for $X$, and $\mathcal{X}$ is the corresponding complexification,
then one can consider the $k$ (generically distinct) points 
$\alpha^1 (\xi) , \dots , \alpha^k (\xi) \in \tilde\Delta_z$, defined for $\xi\in \tilde\Delta_\xi$, satisfying
$(\alpha^j (\xi), \xi) \in \mathcal{X}$ for $j=1,\dots, k$ as 
a point $Z(\xi) = \inp{\alpha^1 (\xi), \dots, \alpha^k (\xi)} \in (\Cn)^k_{\rm sym}$.

Here $X^k_{\rm sym}$ is the $k$-th symmetric power of $X$, i.e. the 
quotient of $X^k$ with respect to the equivalence relation identifying 
two points $(\alpha^1, \dots , \alpha^k)$ with $(\beta^1, \dots, \beta^k)$
if $\beta^j = \alpha^{\ell_j}$ for some permutation $j \mapsto \ell_j$ of 
$\left\{ 1,\dots, k \right\}$. We shall write $\alpha \in \inp{\alpha^1, \dots , \alpha^k}$
if $\alpha = \alpha^j$ for some $j$. 

The variety $\mathcal{X}$ can be considered as  the {\em multigraph} of the 
holomorphic {\em multifunction} $Z(\xi)$ in $\Delta_z \times \Delta_\xi$, meaning
that for any symmetric function 
$h\colon (\Cn)^k \to \C$ the  composition
\[ h(Z(\xi)) = h(\alpha^1 (\xi), \dots, \alpha^k(\xi)) 
\in \holmaps (\Delta_\xi) \]
is holomorphic in $\Delta_\xi$.

The elementary symmetric functions on $(\Cn)^k$ are the 
coefficients $\phi_{\ell,\gamma}$ of the polynomial 
\[ P_{n,k} (x,u) = \prod_{j=1}^k (x - u \cdot \alpha^j)  = \sum_{\ell=0}^k \left( \sum_{\substack{\gamma\in \N^n\\ |\gamma| = \ell}} \phi_{\ell,\gamma} (\alpha^1, \dots, \alpha^k) u^\gamma \right) x^{k-\ell} . \]
In terms of the elementary symmetric functions 
$P_{\ell,k} \in \C[y^1, \dots y^k]$, 
where the polynomial $P_{\ell,k} (y^1, \dots ,y^k)$ is defined by
\[ \prod_{j=1}^k (x-y^j) = \sum_{\ell = 0}^k P_{\ell,k} (y^1, \dots, y^k) x^{k-\ell}, \]
they can also be expressed through the coefficients of the
 polynomials
\[P_{\ell,k} (u\cdot \alpha^1, \dots, u\cdot \alpha^k) = \sum_{\substack{\gamma\in \N^n\\ |\gamma| = \ell}} \phi_{\ell,\gamma} (\alpha^1, \dots, \alpha^k) u^\gamma\] 
In order to transfer results 
for symmetric functions of $k$ scalar variables
to the multivariate case, note that for any $u\in\Cn$, and 
any  symmetric polynomial $P(y^1, \dots , y^k)$ 
we have that $P (u \cdot \alpha^1, \dots , u\cdot \alpha^k)$
is a symmetric function of $(\alpha^1, \dots ,\alpha^k) \in (\Cn)^k$.
In particular, if we recall that the elementary symmetric polynomials 
$P_{\ell, k}$
of the $k$ variables $y^1, \dots, y^k$ uniquely identify
 $\inp{y^1, \dots, y^k} \in (\C)^k_{\rm sym}$, just as the power sums 
\[ S_{\ell, k} = \sum_{j=0}^k (y_j)^\ell \]
do, we have the following Lemma.

\begin{lemma}\label{lem:powers}
A point $\alpha = \inp{\alpha^1, \dots ,\alpha^k} \in (\Cn)^k_{\rm sym} $ 
is uniquely determined by each of the following: 
\begin{compactenum}[\rm i)]
\item $\varphi_{\ell, \gamma} (\alpha^1, \dots , \alpha^k)$ for $0\leq \ell \leq k$, $|\gamma|=\ell$;
\item $S_\beta(\alpha^1, \dots , \alpha^k) = \sum_{j=1}^k (\alpha^j)^\beta$ $0\leq |\beta|\leq \ell$.
\end{compactenum}
\end{lemma}
\begin{proof}
Let us start with i). If we know the $\varphi_{\ell,\gamma}$ for 
$|\gamma|=\ell$, $0\leq \ell \leq k$, we know $P_{\ell,k} (u\cdot \alpha^1 , \dots , u\cdot \alpha^k)$ for $0 \leq \ell \leq k$ and every $u\in \Cn$.
 Hence, we know 
 $\inp{u\cdot \alpha^1 , \dots , u\cdot \alpha^k} \in (\C)^k_{\rm sym}$ 
 for every $u\in \Cn$, and thus, 
 $\inp{\alpha^1, \dots ,\alpha^k} \in (\Cn)^k_{\rm sym}$. 

 For ii), if we know $S_\beta(\alpha^1,\dots ,\alpha^k)$ for $|\beta|\leq k$,
  we know $S_{\ell,k} (u\cdot \alpha^1, \dots ,u\cdot \alpha^k)$ for
  $0\leq \ell \leq k$ and every $u\in\Cn$, and as before, this 
  means we know 
 $\inp{u\cdot \alpha^1 , \dots , u\cdot \alpha^k} \in (\C)^k_{\rm sym}$ 
 for every $u\in \Cn$, and therefore also 
 $\inp{\alpha^1, \dots ,\alpha^k} \in (\Cn)^k_{\rm sym}$. 
\end{proof}
		
If we look at the polynomial
\[ P_{n,k} (u\cdot z, u) = \sum_{|\gamma| = m} \Phi_{\gamma} (z,\xi) u^\gamma, \]
the $\Phi_\gamma$ are polynomials in $z$ with coefficients
which are  holomorphic in $\xi \in \Delta_\xi$. The equations 
\begin{equation}\label{e:standard} \Phi_\gamma (z,\xi) = 0, \,  |\gamma|=k  
\end{equation}
define $\mathcal{X}$ in $\Delta_z \times \Delta_\xi$ and
will be referred to as 
the {\em standard defining equations} of $\mathcal{X}$. Note that 
\begin{equation}\label{e:coefficientsstandard} \Phi_\gamma (z,\xi) = \frac{k!}{\gamma!} z^\gamma + \sum_{\beta < \gamma} a_\beta^\gamma (\xi) z^\beta. \end{equation}

\begin{example}
The Bishop surface $w = \lambda (z^2 + \bar z^2) + |z|^2$ is 
a germ of a  Segre-nondegenerate hypersurface in $\C^2$ of multiplicity $2$
if $\lambda\neq 0$.
 Its 
standard defining equations are given by 
\[ \begin{aligned}
\Phi_{(2,0)} (z,w, \bar z, \bar w) &= \left(-\frac{\bar w}{\lambda }+\frac{|z |^2}{\lambda
   }+\bar z^2+z^2\right), \\ \Phi_{(1,1)} (z,w, \bar z, \bar w) &= \frac{ \left(w-\bar w\right) \left(\bar z+2 \lambda 
   z\right)}{\lambda }, \\ \Phi_{(0,2)} (z,w, \bar z, \bar w) &= \left(w-\bar w\right)^2.
\end{aligned}
\]
\end{example}

We finally note that by Proposition~\ref{prop:reflection} we can 
also consider the ``barred'' defining equations 
(in a possibly smaller polydisc), where $z$ and $\xi$ change their 
roles, that is, the functions 
\begin{equation}\label{e:barredstandard}\overline{\Phi_\gamma} (\xi ,z) , \quad |\gamma| = k,\end{equation}
are also (standard) defining equations for $\mathcal{X}$ (considered 
as the complexification of $(X,0)$ in a possibly smaller neighbourhood of the 
origin). 

\subsection{The Segre varieties as multifunctions} 
\label{sub:the_segre_sets_as_multifunctions}
Given $(X,0)\subset \Cn$, 
the Segre variety $S_q$ of a point $\bar q\in {\tilde\Delta_\xi}$ with 
respect to the $\pi_2$-good polydiscs $\tilde\Delta_z \times \tilde\Delta_\xi$ is usually defined by  
\[ S_q = \left\{ z \in \tilde\Delta_z \colon (z,\bar q) \in \mathcal{X} \right\}, \]
where $\mathcal{X}$ is, as before, the complexification of $(X,0)$ 
with respect to $\tilde\Delta_z \times \tilde\Delta_\xi$.
For generic $q$, we have that $|S_q| = k$. 
If we would like the  
map $q \mapsto S_q$ to 
be antiholomorphic, we need 
to identify it with the holomorphic 
multifunction $Z \colon \tilde\Delta_\xi \to (\tilde\Delta_z)^k_{\rm sym}$
defined by 
\begin{equation}
	\label{e:multi1}  Z(\xi) = \inp{z^1 (\xi), \dots ,z^k(\xi)}, \quad \xi\in \tilde\Delta_\xi,
\end{equation}
where $Z(\xi) = (\pi_2|_{\mathcal{X}})^{-1} (\xi)$ via $S_q = Z(\bar q)$. We can  equivalently
consider $\Xi \colon \Delta_z \to (\Delta_\xi)^k_{\rm sym}$,
\begin{equation}
	\label{e:multi2}  \Xi (z) = \inp{\xi^1 (z) , \dots , \xi^k (z)}, \quad z\in \Delta_z, 
\end{equation}
where $\Xi(z) = (\pi_1|_{\mathcal{X}})^{-1} (z)$
for suitable polydiscs $\Delta_z \times \Delta_\xi$; 
by Proposition~\ref{prop:reflection}, these multifunctions (as germs at $0$) are
related by $Z = \overline{ \Xi}$ (or equivalently $\Xi = \overline{Z}$, where we define
the barred multifunction by $\overline{Z (\xi)} = \overline{Z} (\overline{\xi})$.

By Proposition~\ref{prop:reflection}, 
these two maps have additional properties which
are usual for 
Segre varieties. We combine this
with the basic invariance result for 
holomorphic maps: note that if 
$H\colon (\Cn,0) \to (\C^{n'},0)$ 
is a germ of a holomorphic map satisfying
$H((X,0)) \subset (Y,0)$, then the map
$\mathcal{H} (z,\xi) = (H(z) ,\bar H(\xi))$
satisfies $\mathcal{H} ((\mathcal{X},0) \subset (\mathcal{Y},0)$, where $\mathcal{Y}$
denotes the complexification of $(Y,0)$. 

\begin{lemma}
\label{lem:propertiessegre} Let $(X,0)\subset (\Cn,0)$, $(Y,0) \subset (\C^{n'},0)$
be germs of  Segre-nondegenerate 
real-analytic varieties at the origin,  $\tilde\Delta_z \times \tilde\Delta_\xi$ be 
$\pi_2$-good for $(X,0)$, and $\tilde\Delta_{z'} \times \tilde\Delta_{\xi'}$ be $\pi_2'$-good for $(Y,0)$. Also,
let $H\colon (\Cn,0) \to (\C^{n'},0)$ 
be a germ of a holomorphic map satisfying
$H((X,0)) \subset (Y,0)$. Let $Z$ and $Z'$
be defined as above. 
Then the following hold
(provided all of the expressions involved are defined, and the good polydiscs are small enough): 
\begin{compactenum}[\rm i)]
\item $p \in Z(\bar q)$ if and only if $q \in Z(\bar p)$.
\item $z \in Z(\bar z)$ if and only if 
$z\in (X,0)$.
\item $H(p) \in Z'(\overline{H(q)})$ if $p\in Z(\bar q)$.
\end{compactenum}
\end{lemma}
\begin{proof}
The content of the Lemma are
just convenient restatements of the 
fact that $Z(\xi) = \pi_z^{-1} (\xi)$ 
(and similarly for $Z'$). So in 
order to prove i), $p\in Z(\bar q)$ 
means that $(p,\bar q) \in \mathcal{X}$,
which in turn means that $(q, \bar p) \in \mathcal{X}$, i.e. $q\in Z(\bar p)$ (provided
the latter is defined). ii) follows similarily. 

For iii), we appeal to the 
fact that $\mathcal{H} ((\mathcal{X},0)) \subset (\mathcal{Y},0)$, which means that 
$p\in Z(\bar q)$ implies $(H(p),\overline{H(q)}) \in \mathcal{Y}$, i.e. $H(p) \in Z'(\overline{H(q)}) $.
\end{proof}

\subsection{Real-analytic and formal functions on $(X,0)$} 
\label{sub:real_analytic_and_formal_functions_on_}
A function $f \colon X \cap U \to \C$ is said to be real-analytic 
on $X\cap U$ if there exists a neighbourhood $\tilde U$ of $X$ in $\C^{n}$
and a real-analytic function $\tilde f \colon \tilde U \to \C$ on $\C^{n}$
such that ${\tilde{f}}|_{X\cap U} = f$. The set $\diffable{\omega} (X,0) $
of germs of real-analytic 
functions on $X$ at $0$ is therefore naturally identified with the local
function ring of $X$ at $0$, which we are going to denote by 
\[ \cps{X} = \faktor{\cps{z,\bar z}}{\idealsheaf_0 (X)}. \]
We also define the ring of formal functions of $X$ by 
\[ \fps{X} = \faktor{\fps{z,\bar z}}{\hat{\idealsheaf}_0 (X)}, \]
where $\hat{\idealsheaf}_0 (X) = \fps{z,\bar z} \idealsheaf_0 (X)$ is 
the ideal generated by $\idealsheaf_0 (X) \subset \fps{z,\bar z}$ in 
the ring of germs of formal power series. Note that $\hat{\idealsheaf}_0 (X)$
is again a real ideal. The corresponding rings in the complexification 
are 
\[ \cps{\mathcal{X}} =\faktor{\cps{z,\xi}}{\idealsheaf_0 (\mathcal{X})}, 
\quad \fps{\mathcal{X}} = \faktor{\fps{z,\bar z}}{\hat{\idealsheaf}_0 (\mathcal{X})}.\]
Note that the rings $\cps{\mathcal{X}}$ and $\cps{X}$ as well as 
$\fps{\mathcal{X}}$ and $\fps{X}$ are isomorphic; however, it is convenient
to distinguish between $X$ and its complexification $\mathcal{X}$.

An important remark, echoing Proposition~\ref{p:genericallytotallyreal} 
in the formal setting, is that the natural maps 
\[ \fps{z} \to \fps{X}, \quad \fps{z} \to\fps{\mathcal{X}},\quad \fps{\bar z} \to \fps{X}, \quad\fps{\xi} \to \fps{X},    \]
are all {\em injections}:
\begin{lemma}\label{lem:formalsubset}
If $(X,0)$ is a Segre-nondegenerate germ, then there 
are natural inclusions
\[\fps{z} \subset \fps{X}, \quad \fps{z} \subset\fps{\mathcal{X}},\quad \fps{\bar z} \subset \fps{X}, \quad\fps{\xi} \subset \fps{X}.\]
\end{lemma}

\begin{proof}
We only prove this for $\fps{\bar z}\subset \fps{X}$, the other assertion being either direct consequences or analogous to that. 
Since $X$ is Segre-nondegenerate, if we choose a set of generators
$\varrho^1,\dots , \varrho^d $ of $\idealsheaf_0 (X)$, the map 
$\Cn \ni z \mapsto H(z) = \left( \varrho^1 (z,0),\dots , \varrho^d(z,0) \right) \in \C^d $
is finite. Therefore, the matrix 
$ \dopt{H}{z} (z) $ is generically of full rank; hence 
the matrix $\dopt{\varrho}{z} (z,\bar z)$ is also. If now a function 
of the form $\varphi(\bar z)$ is congruent to $0$ mod $\idealsheaf_0 (X)$, 
we can write 
\[  \varphi(\bar z) = \sum_{j=1}^d \varphi_j (z,\bar z) \varrho^j (z,\bar z).  \]
Taking a derivative with respect to $z$ yields
\[ 0 = \sum_j \varphi_j (z,\bar z) \varrho^j_z (z,\bar z)  + \sum_j \varphi_{j,z} (z,\bar z) \varrho^j (z,\bar z)
, \]
and so (since the $\varrho^j_z (z,0)$ are generically independent)
we have that $\varphi_j(z,0) = 0$.

Now assume that we know that $\varphi_{j,\bar z^\alpha} (z,0) =0$ for 
$|\alpha| < k$. We take  a $\beta\in \N^n$  with $|\beta| = k$ and compute
\[ \dopt{^{|\beta|} \varphi}{\bar z^\beta}  = \sum_j \varphi_{j,\bar z^\beta} \varrho^j + P(\varphi_{j,\bar z^\alpha} \colon |\alpha| <k). \]
Taking the derivative with respect to $z$ once again and evaluating at
$\bar z = 0$ yields $\varphi_{j,\bar z^\beta} (z,0) = 0$. By induction, 
we therefore have $\varphi_j (z,\bar z) = 0$, $j=1, \dots, d$, and 
hence $\varphi =0$.
\end{proof}

\subsection{Remark: Formal varieties} 
\label{sub:remark_formal_manifolds}
One can, instead of working with ideals coming from an actual 
manifold $X$, often also obtain results which are valid 
for {\em formal Segre-nondegenerate varieties}: 
\begin{defn}
\label{def:formalideal} A radical ideal
 $\hat{\mathcal{I}} \subset \fps{z,\bar z}$  is said to define 
 a formal Segre nondegenerate variety $\hat X$ (at $0$) if 
 \begin{compactenum}
 \item $\hat{\mathcal{I}}$ is real, i.e. $\sigma (\hat{\mathcal{I}}) \subset \hat{\mathcal{I}}$, and 
 \item  $S = \left\{ \varrho(z,0) \colon \varrho \in \hat I \right\}$ is 
 an ideal of definition, i.e. there exists a $k$ such that the maximal ideal $\hat{\mathfrak{m}} \subset \fps{z,\bar z}$ satisfies $\hat{\mathfrak{m}}^k \subset \hat{\mathcal{I}}$. 
 \end{compactenum}
 The Segre multiplicity of the formal variety $\hat X$ is 
 defined to be $\dim_\C \faktor{\fps{z}}{S}$.
\end{defn}

We will use the same notation for formal varieties $\hat X$ (defined 
by the ideal $\hat{\idealsheaf}_0 (\hat X)$) that
we introduced real-analytic varieties above.  In particular, we would like to point out that 
Lemma~\ref{lem:formalsubset} holds (with the same proof) 
in this setting: i.e. we have that 
\[ \fps{z} \subset \fps{\hat X}, \quad \fps{z} \subset\fps{\hat{\mathcal{X}}},\quad \fps{\bar z} \subset \fps{\hat{X}}, \quad\fps{\xi} \subset \fps{\hat{X}}.  \]

\section{Obstructions to holomorphicity} \label{sec:obstr}

It is well known that for a real-analytic 
maximally real submanifold $E$, every real-analytic function 
on $E$ is the restriction of a holomorphic function in a neighbourhood 
of $E$. In the presence of a singularity, this is not necessarily the 
case any longer, as the following 
simple  example shows. 

\begin{example}\label{exa:sillyexample}
Consider the variety $X$ in $\C^2_{(z,w)}$ defined by 
\[ z^3 = w^2 + \bar w^2. \] This 
variety is Segre-nondegenerate, of Segre multiplicity $6$, and 
one checks that the real-analytic function 
$\bar w|_X$ is not the restriction of any holomorphic function 
in $\C^2$ to $X$. Indeed, if it were, say $\bar w|_X = f(z,w)$, then
by Proposition~\ref{p:genericallytotallyreal} we would have 
$f(z,w)^2 = z^3 - w^2$ as germs at the origin, which is absurd.
\end{example}

However, the Segre multiplicity of $(X,0)$ gives a rough bound 
for how many nonholomorphic real-analytic functions there are. 

\begin{prop}\label{p:polynomialrepresentation}
Let $(\hat X,0) \subset (\C^n,0)$ be a formal
Segre nondegenerate
subvariety, of Segre multiplicity $k$.  
Then there exists an operator 
\[T\colon \fps{z,\bar z} \to \fps{z}[\bar z]\]
valued in the space of polynomials of degree at most $k$ in $\bar z$ 
such that for every formal power series $f \in \fps{z,\bar z} $, 
we have that
 $Tf(z,\bar z) \in  \fps{z}[\bar z]$ 
 is a representative for the class of $f$ in $\fps{\hat X}$.
 Furthermore, if $X$ is a real-analytic variety and
if $f\in \cps{X}$, then $Tf \in \cps{z,\bar z}$.
\end{prop}

\begin{proof} We use the standard defining equations introduced in \eqref{e:standard} above in their conjugate versions
in an adaptation of standard Weierstrass division. So let  
\[ \bar \Phi_\gamma (\bar z, z) = \frac{k!}{\gamma!} \bar z^\gamma + \sum_{\beta< \gamma} a_\beta^\gamma (z) \bar z^\beta \in \hat{\idealsheaf}_0 (\hat X), \]
where the $a_\beta^\gamma (z) \in \fps{z}$ vanish at $0$. For any 
formal power series $\varphi(z,\bar z)$, we write 
\[ \varphi(z,\bar z)= \varphi_0 (z,\bar z) + \sum_{|\gamma| =k} \bar z^\gamma T_\gamma \varphi(z,\bar z)  \]
with  $\varphi_0 (z,\bar z)$ a polynomial of degree at most $k$ in $\bar z$ and some choice of $T_\gamma$. Consider the operator $S\colon \fps{z,\bar z}\to \fps{z,\bar z}$ defined 
by 
\[ (S\varphi) (z,\bar z) = \varphi_0 (z,\bar z) +  
\sum_{|\gamma| =k} \frac{\gamma!}{k!} \bar \Phi_\gamma(\bar z, z) T_\gamma \varphi(z,\bar z).\]
Then $(I-S)\varphi$ vanishes to order strictly exceeding the order
 of vanishing of $\varphi$ at $0$. It follows that $S$ is bijective, 
 its inverse given by $S^{-1} f = \sum_j (I - S)^j f$. 
 Given a formal power series $f(z,\bar z)$, denote by $\varphi(z,\bar z) = (S^{-1} f)(z,\bar z)$. 
 We then have 
 \[ f(z,\bar z) = {(S\varphi)}(z,\bar z) = \varphi_0 (z,\bar z) +  
\sum_{|\gamma| =k} \frac{\gamma!}{k!} \bar \Phi_\gamma(\bar z, z) T_\gamma \varphi(z,\bar z).  \]
Therefore, $f(z,\bar z)$ and $\varphi_0$ (which is a polynomial of 
degree at most $z$) agree modulo the standard defining equations of $\hat X$.

For the proof of convergence, we need a little preparation, even though it
 is very similar. In that case, $X$ is given
by real-analytic equations, and hence the $a_\beta^\gamma (z)$ converge 
in a neighbourhood of the origin. We denote by $\Delta^p_r$ and $\Delta_s^q$ two 
polydiscs (of arbitrary polyradius $r$ and $s$, respectively) in 
$\C^p$ and $\C^q$, respectively. We first prove the following  

{\bf Claim:} For $p$, $q$, and $k\in\N$, there exist operators 
$T_\gamma\colon \holmaps^\infty (\Delta_r^p \times \Delta_s^q)$ for $\gamma\in \N^q$
 with $|\gamma| = k$ 
such that 
\[ \varphi(z,w) = \varphi_0 (z,w)  + \sum_{|\gamma| = k} w^\gamma T_{\gamma}\varphi (z,w) \]
where $\varphi_0 (z,w)$ is a polynomial of degree less than $k$ in $w$, 
and such that with $\vnorm{\cdot}_K$
denoting the supremum norm over a set $K$, 
\[ \vnorm{T_\gamma \varphi (z,w)}_{\Delta_r^p \times \Delta_s^q} \leq \left( \frac{2}{s} \right)^{\gamma}  \vnorm{\varphi(z,w)}_{\Delta_r^p \times \Delta_s^q}.  \]

The proof of this claim is by induction on $q$. 
For $n=1$, 
we write 
\[ \varphi (z,w_1 ) = \varphi(z,0) + {w_1} (T \varphi) (z,w), \]
where $T \varphi (z,w_1) = \frac{\varphi(z,w_1) - \varphi(z,0)}{w_1}$ clearly 
satisfies 
\[ \vnorm{T \varphi}_{\Delta_r \times \Delta_s} \leq \frac{2}{s_1} \vnorm{\varphi}_{\Delta_r \times \Delta_s}. \]
We can thus write 
 \[ \varphi (z , w_1) = \sum_{j=0}^{k-1} \varphi_j (z) w_1^j + w_1^k (T^k \varphi) (z,w_1),\]
 with 
 \[\vnorm{\varphi_j}_{\Delta_r^p \times \Delta_{s_1}} \leq \left( \frac{2}{s_1} \right)^j \vnorm{\varphi}_{\Delta_r^p \times \Delta_{s_1}}, \quad 
  \vnorm{ (T^k \varphi) (z,w_1)}_{\Delta_r^p } \leq \left( \frac{2}{s_1} \right)^k \vnorm{\varphi}_{\Delta_r^p \times \Delta_{s_1}}. \]

  Assuming that we know the corresponding estimates in dimension $q-1$, we write 
  $w =  (w',w_n) \in \C^{n-1} \times \C$ and (as before in the one variable 
  case) obtain 
  \[ \varphi (z, w) = \sum_{j=0}^{k-1} \varphi_j (z,w') w_n^j + w_n^k (T^k \varphi) (z,w), \]
  with 
  \[\vnorm{\varphi_j}_{\Delta_r^p \times \Delta_{s'}^{q-1}} \leq \left( \frac{2}{s_n} \right)^j \vnorm{\varphi}_{\Delta_r^p \times \Delta_s^q}, \quad 
  \vnorm{ (T^k \varphi) (z,w_1)}_{\Delta_r^p \times \Delta_s^q} \leq \left( \frac{2}{s} \right)^k \vnorm{\varphi}_{\Delta_r^p \times \Delta_s^q}. \]
  We thus define $T_{(0,\dots, 0, k)} := T^k$ and for $\ell < k$,
  \[ T_{(\gamma', \ell)} \varphi (z,w) = T_{\gamma'} \varphi_{\ell} (z,w') ,\]
  which satisfies all of our requirements. 

  With the claim, we can now set up the operator $S$ as in the formal part of 
  the proof, but as an operator between $\holmaps^\infty (\Delta_r^n \times \Delta_s^n)$, for 
  small $(r, \dots ,r)$ and $(s,\dots, s)$. When estimating $I-S$, one now obtains 
  that 
  \[ \vnorm{(I-S) \varphi}_{\Delta_r^n \times \Delta_s^n} \leq C \frac{\max_{\gamma,\beta} \vnorm{a_\gamma^\beta}_{\Delta_r^p}}{s^k},\]
  which for small $r$ is small, and thus, for such $r$ this operator is invertible (as an endomorphism
  of the Banach space $\holmaps^\infty (\Delta_r^n \times \Delta_s^n)$). The rest of 
  the proof is then exactly as before.
\end{proof}


Recall from \S~\ref{sub:remark_formal_manifolds}
that $\fps{z} \subset \fps{\hat X}$. The preceding Proposition says 
in particular 
that 
\[ \dim_\C \faktor{\fps{\hat X}}{\fps{z}} \leq \binom{n+k-1}{k-1}  \]
is finite.

 The next
theorem characterizes holomorphy of a germ $f(z,\bar z)$ in 
the sense that  a real-analytic function 
on a real-analytic Segre nondegenerate variety comes from the 
restriction of a holomorphic function if (and only if) its complexification is 
constant along the fibers of the projection on the first coordinate.

\begin{thm} \label{thm:constantonsegre}
Let $(X,0) \subset (\C^n,0)$ be an irreducible germ of 
a Segre nondegenerate $n$-dimensional
real-analytic subvariety.
Let $f\in \diffable{\omega} (X,0)$ be represented 
by $f(z,\bar z)\in \cps{z, \bar z}$, so that $f$ is holomorphic
on $\Delta_z \times \Delta_\xi$, which we assume to 
be $\pi_1$-good for $(X,0)$. Denote by 
$\sX \subset \Delta_z \times \Delta_\xi$
the corresponding complexification.
Then
there exists a holomorphic function $F \colon \Delta_z \to \C$
such that $f(z,\bar{z}) = F(z)$ for  $z \in X$ close 
by the origin
if and only if there exists an open set $U \subset \Delta_z$
such that
\begin{equation}
f(z,\xi_1) = f(z,\xi_2)
\end{equation}
whenever $z \in U$ and $(z,\xi_1) \in \sX$ and $(z,\xi_2) \in \sX$.
\end{thm}

\begin{proof} We only need to prove the ``if'' part. 
As $\sX$ is irreducible, then the regular part $\sX_{\rm reg}$
is connected.
Suppose $\pi_1$ is generically $k$ to 1.
Let $V \subset \sX_{\rm reg}$ be the set where
$\pi_1$ is $k$ to 1, that is $V = \sX_{\rm reg} \setminus \pi_1^{-1}(D)$
where $D$ is the discriminant set of the projection, a complex analytic
subvariety.
Taking two preimages $\xi_1$ and $\xi_2$ as functions of $z$,
then for any $z_0 \in \Delta_z \setminus D$
we can analytically continue $f(z,\xi_1(z))-f(z,\xi_2(z))$
until we get to a $z \in U$.
As $U$ is an open set and
$f(z,\xi_1(z))-f(z,\xi_2(z))$ is identically zero on $U$, we find
that $f(z_0,\xi_1(z_0))-f(z_0,\xi_2(z_0))$.
Or in other words,
$f(z,\xi_1) = f(z,\xi_2)$ whenever $z \in \Delta_z \setminus D$
and $(z,\xi_1)$ and $(z,\xi_2)$ are in $\sX$.

Thus for all $z \in \Delta_z \setminus D$,
define $F(z) = f(z,\xi)$ for some $\xi$ such that $(z,\xi) \in \sX$.
Clearly $F$ is well defined by the above argument.
$F$ is locally bounded as $f$ is locally bounded in
all of $\Delta_z \times \Delta_\xi$.
Further $D$ is a subvariety of $\Delta_z$,
and so by the Riemann extension theorem, $F$ is a holomorphic function
on $\Delta_z$.
\end{proof}

In the next section we will introduce the ``right'' formulation of 
the preceding Theorem in order to be able to make it into a formal statement as 
well. 

\section{The averaging operator} \label{sec:averaging}

Let $(X,0) \subset (\C^n,0)$ be an irreducible germ at 0 of Segre-nondegenerate $n$-dimensional
subvariety, of Segre multiplicity $k$.
Let $\Delta_z \times \Delta_\xi$
be good for $(X,0)$, and let
$\sX \subset \Delta_z \times \Delta_\xi$
be the corresponding complexification.

%

In the previous
section we proved that a function on $X$ that complexifies to
$\Delta_z \times \Delta_\xi$ is a restriction of a holomorphic function if
and only if $\xi \mapsto f(z,\xi)$ is constant on $(\pi_1|_{\mathcal{X}})^{-1}(z)$.
We may therefore average out an arbitrary real analytic function
to obtain a holomorphic function as follows.



We recall the multifunctions $Z(\xi)$ and 
$\Xi (z)$, which are defined near the origin with 
values in $(\Delta_z)^k_{\rm sym}$ and $\tilde \Delta_\xi$, respectively, 
by \eqref{e:multi1} and \eqref{e:multi2}.
Given a germ of a real-analytic function $f \in \cps{z, \bar z}$, assume
that $f$ extends to be a holomorphic function on $\Delta\times \overline{\Delta}$, so that 
\begin{equation}
g(z,\inp{\omega^1,\ldots,\omega^k}) = \frac{1}{k} \sum_{j=1}^k f(z,\omega^j) ,
\end{equation}
 is a holomorphic function on $\Delta \times (\overline{\Delta})^k_{\mathrm{sym}}$. 
 For a suitable neighbourhood polydisc $\tilde \Delta \subset \Delta$, 
we can assume that $\Xi(\tilde \Delta) \subset (\overline{\Delta})^k_{\mathrm{sym}} $.

Hence we can define a holomorphic function 
$\averaging f \colon \tilde \Delta \to \C$
by
\begin{equation}\label{e:averagingdefined}
(\averaging f)(z) = g \left( z, \Xi (z) \right) =  \frac{1}{k} \sum_{j=1}^k f(z,\xi^j (z))
\end{equation}

This definition gives a (linear) map 
$\averaging \colon \cps{z, \bar z} \to \cps{z}$, and the reader can 
 check that $\averaging \colon \fps{z, \bar z} \to \fps{z}$ can 
be defined for formal power series, even if $\hat X$ is merely assumed to be formal. 

The next result summarizes the properties which show that 
 the operator $\averaging$ encodes the 
obstruction to holomorphic extension of a real-analytic function on
$(X,0)$.
\begin{lemma}\label{thm:averagingop} Let $( X, 0)$ be a germ of a real-analytic Segre nondegenerate 
variety. 


The mapping $\averaging$ has $\idealsheaf_0 (X) \subset \ker \averaging$.
In particular, it descends to a map (again denoted by the same letter)
$\averaging \colon \diffable{\omega}(X,0) \to \cps{z}$.
 



A function $f \in \diffable{\omega} (X,0)$ is the 
restriction of 
a germ of a holomorphic function on $\Cn$ if 
and only if  $\averaging f|_X = f|_X$.
\end{lemma}

\begin{proof} We have already discussed linearity. If $\varrho \in \cps{z, \bar z}$ vanishes on $(X,0)$, its complexification 
vanishes on $\mathcal{X}$, and hence $\averaging \varrho =0$.
For the last statement, we only need to prove the necessity of 
the given characterization. 
By Theorem~\ref{thm:constantonsegre}, $f(z,\xi_1) = f(z,\xi_2)$ when
$(z,\xi_1), (z,\xi_2) \in \mathcal{X}$.    
It follows that $\averaging f|_X = f|_X$. 
\end{proof}

We note that together with Proposition~\ref{p:polynomialrepresentation},
this gives a rather complete picture of the obstructions
to holomorphicity: they are encoded in the behaviour of $\averaging$
on functions of the form $\bar z^\alpha$ for $|\alpha|<k$. Before 
we discuss this fact further, we give some examples. 

\begin{example}
\label{exa:averagingbishop} We again consider the Bishop surface 
$w = \lambda (z^2 + \bar z^2) + z \bar z$. The function $\Xi(z)$
is computed to be 
\[ \Xi(z) = \left\langle \left( \frac{ - z + \sqrt{-4 \lambda ^2 z^2+z^2+4 w
   \lambda }}{2 \lambda }, w \right) , \left( \frac{ -z -\sqrt{-4 \lambda ^2 z^2+z^2+4 w
   \lambda }}{2 \lambda } ,w\right)   \right\rangle . \]
One therefore computes that 
\[ \averaging \bar z = \frac{-z}{2 \lambda}, \quad \averaging \bar w = w.\]
   Note that 
   the computation of
   \[ \averaging \bar z^2 =\frac{2 \lambda  w+ (1-2 \lambda ^2)
   z^2}{\lambda ^2}.\] can be done
    using the defining relation 
   of the surface as well as direct 
   application of the definition of 
   the averaging operator. We will return to this 
   observation in more generality below. 
\end{example}

First, the averaging operator depends only on values on $X$, which is clear
since $\sA f$ is holomorphic, but we have more.
The following proposition replaces equality of two real-analytic
functions on $(X,0)$ by equality of two germs of holomorphic functions
$(\C^n,0)$.  In particular, since $\averaging$
is defined canonically, we no longer need to consider the defining functions
of $X$.  Furthermore this equality can now be done formally.

\begin{prop}\label{pro:detbypowers}
Suppose $(X,0) \subset \C^n$ is an irreducible germ of a Segre nondegenerate
$n$-dimensional real-analytic subvariety of multiplicity $k$.
Let $f$ and $g$ be germs of real-analytic functions.
Then $f|_X = g|_X$ if and only if
$\averaging ((f-g)^\ell) = 0$ for all $\ell = 1,\ldots,k$.
\end{prop}

\begin{proof}
If $f|_X = g|_X$, then also their complexifications are
equal on the complexification of $X$.  For a point on $X$,
$\averaging f$ only depends on the values of $f$ on the complexification of
$X$, and so $\averaging f = \averaging g$; the same argument holds for all 
powers of $f$ and $g$. 

On the other hand suppose
$\averaging ((f-g)^\ell) =0$ for all $\ell = 1,\ldots,k$.
Consider a good neighbourhood $\Delta_z \times \Delta_\xi$ for $X$.
We can assume that $X$ is a closed subset of $\Delta_z$.
Fix a point $z \in X$.  By the definition of the averaging operator
and the hypothesis of the proposition,
\begin{equation}
\sum_{j=1}^k {
\left(f(z,\xi^j (z))-g(z,\xi^j (z))\right)}^\ell =
0
\end{equation}
The power sums for powers $1$ through $k$
uniquely determine an unordered set of $k$ complex numbers.  Hence
it must be that 
$f(z,\xi^j(z)) = g(z,\xi^j (z))$ for all $j$.  Since $z \in X$, then
for at least one $j$, we have $\xi_j(z) = \bar{z}$ and therefore,
$f(z,\bar{z}) = g(z,\bar{z})$.
\end{proof}

Similarly, we can check if a function is a restriction of a holomorphic
function.

\begin{thm}
Suppose $(X,0) \subset \C^n$ is an irreducible germ of a Segre nondegenerate
$n$-dimensional real-analytic subvariety of multiplicity $k$.
Suppose that $f(z,\bar{z})$ is a
real-analytic function.  Then $f|_X \in \cps{z}$ 
if and only if
$\averaging (f^\ell) = {(\averaging f)}^\ell$ for all $\ell = 1,\ldots,k$.
\end{thm}

\begin{proof}
If $f|_X = h|_X$ for a holomorphic $h$, then
$f^\ell|_X = h^\ell|_X$ for all $\ell$ and therefore
$\averaging(f^\ell) = h^\ell = \averaging(f)^\ell$ for all $\ell$.

On the other hand, suppose
$\averaging (f^\ell) = {(\averaging f)}^\ell$ for all $\ell = 1,\ldots,k$
and let 
$h = \averaging f$.  Then 
\[ \averaging (f-h)^\ell = \sum_{j=0}^\ell (-1)^j \binom{\ell}{j}
h^{\ell-j} \averaging (f^j) = 0,  \]
and so by Proposition~\ref{pro:detbypowers}, $f= \averaging f$, i.e. 
$f$ is holomorphic. 
\end{proof}







\section{Flattening and other applications of averaging}
\label{sec:flattening}

Even though we introduced the averaging operator for arbitrary 
real-analytic function germs on $(X,0)$, 
it is completely determined by its action on the subspace 
$\cps{\bar z} \subset \diffable{\omega} (X,0)$. 
We therefore define the {\em restricted averaging operator}
\begin{equation}
\restricted = \averaging|_{\cps{\bar z}} \colon \cps{\bar z} \to \cps{z}.
\end{equation}
The operator $\averaging$ can be recovered from $\restricted$ since
\begin{equation}
\averaging(z^\alpha \bar{z}^\beta)
= z^\alpha \averaging(\bar{z}^{\beta})
= z^\alpha \restricted(\bar{z}^{\beta}) .
\end{equation}
We can use the restricted averaging operator to characterize another 
important property of a real-analytic variety $(X,0)$, namely, 
whether it is {\em flattenable}. We will say that $(X,0)$ can 
be flattened if there exists a germ of a holomorphic function $f(z) \in \cps{z}$ 
such 
that $f|_X$ has real values. We shall write $\bar f \in \cps{\bar z}$ for 
the antiholomorphic function defined by $\bar f ( \bar z ) = \overline{f(z)}$.
The terminology is explained by the fact that $(X,0)$ can be flattened 
if and only if there exists a (possibly singular) Levi flat hypersurface, defined
by $\Im f = 0$, containing $(X,0)$. 
Also being flattenable is encoded in the 
averaging operator. 



\begin{thm}\label{thm:realhol}
Suppose $(X,0) \subset \C^n$ is an irreducible germ of a Segre nondegenerate
$n$-dimensional real-analytic subvariety of multiplicity $k$.
Suppose that $f(z,\bar{z})$ is a
real-analytic function.  Then a holomorphic function $f$ is real-valued
on $X$ (that is, $X$ is flattenable) if and only if
$\restricted (\bar{f}^\ell) = f^\ell$ for all $\ell = 1,\ldots,k$.
\end{thm}

\begin{proof}
If $f$ is a holomorphic function that is real-valued on $X$,
then $\bar{f}^\ell|_X = f^\ell|_X$ for all $\ell$.  Hence
$\restricted (\bar{f}^\ell) = \restricted (f^\ell) = f^\ell$ for all $\ell$.

On the other hand, suppose that
$\restricted (\bar{f}^\ell) = f^\ell$ for all $\ell = 1,\ldots,k$.
We can then compute that 
\[ \averaging(f - \bar f)^\ell = 
\sum_{j=0}^\ell (-1)^j \binom{\ell}{j} f^j \restricted(\bar f^{\ell-j}) =0. \]
\end{proof}



As a final application of the restricted averaging operator $\restricted$, we note that it  contains all the information necessary to define $X$.

\begin{thm}
\label{thm:uniquedet} Let $M_1$, $M_2$ be germs of  Segre-nondegenerate real-analytic subvarieties of multiplicity $k$, and let $\restricted_1$, $\restricted_2 $ be their restricted averaging operators. If $\restricted_1 \bar z^\beta = \restricted_2 \bar z^\beta$ for every $\beta$ with $|\beta|\leq k$, then $M_1 = M_2$. 
\end{thm}

\begin{proof}
We claim that given the restricted averaging operator of any such variety $M$ allows
us to reconstruct the defining equation. The restricted averaging operator 
gives us the power sums of all monomials $\bar z^\beta$ for $|\beta|\leq k$, from which we 
can explicitly compute the elementary symmetric  functions (see e.g. \cite{MR3443860}) 
needed to construct the 
standard defining equations \eqref{e:standard}. Thus, $M$ is uniquely determined by 
its restricted averaging operator. 
\end{proof}

\section{Examples of flattening} \label{sec:examplesofflat} 

\label{sec:flattening_cubic_models_in_}
We are now going to consider in some detail models of the form
\[ M_p \colon w = p(z, \bar z) = \sum_j \alpha_j z^j {\bar z}^{k-j}, \]
in particular, for the cases $k=2,3$. These were considered  
by Moser--Webster~\cite{MR709143} for $k=2$ 
and 
Harris~\cite{MR773068} for $k\geq 2$ before. Harris showed that one can consider them 
as the lowest order invariant (in a suitable sense) at a CR singular 
point of a codimension $2$ submanifold of $\C^2$.

We shall, for our purposes, introduce weights $k$ for 
$w$ and $1$ for $z$ and use a bit of a different normalization 
than Harris did, adapted in 
particular to the case that we are interested in, 
namely, that $M_p$ is 
Segre nondegenerate, which means that $p(z,0) \nequiv 0$ in our setting. 

This allows us to  choose coordinates in such a way 
that $\alpha_k = 0$ and $\alpha_0 = 1$. This fixes coordinates $(z,w)$ for $M_p$ up to a finite group of rotations in $z$, unless $p(z,\bar z) = \bar z^{2 k} + \alpha_k |z|^{2 k} $; in that 
case, we can normalize further so that $\alpha_k \geq 0$.

The averaging operator for 
$M_p$ will be denoted by $\averaging_p$, and the restricted averaging operator for $M_p$ will be denoted 
by  $\restricted_p$. We also denote the space 
of (weighted) bihomogeneous polynomials of degree $a$ in $(z,w)$
and $b$ in $(\bar z, \bar w)$ by
\[ \hompoly_{a,b} := \left\{ p \in \C[z,w,\bar z, \bar w] \colon
p(t z, t^k w, s \bar z, s^k \bar w) = t^a s^b p(z,w,\bar z, \bar w) \right\} .\]

We then have 
\begin{lemma}
\label{lem:averagehom} The averaging operator $\averaging_p$ maps 
weighted homogeneous polynomials to weighted homogeneous
polyonomials: 
\[ \averaging_p \colon \hompoly_{a,b}\to \hompoly_{a+b, 0}, 
\quad \restricted_p \colon \hompoly_{0,b} \to \hompoly_{b,0}.  \] 
\end{lemma}

\begin{proof}
It is enough to show that the restricted averaging operator
satisfies the claim. Denoting as usual by $\xi^1 (z,w) , \dots , \xi^k (z,w) $
the points $(z,w,\xi^j(z,w)) \in \mathcal{X}$, where the 
complexification is now defined on all of $\C^4$, we write
$\xi^j = (\zeta^j, \eta^j)$ and first find 
that the elementary symmetric functions $s^1, \dots , s^k$ of the $\zeta^j$
satisfy
\[ s^k (\zeta^1 (z,w), \dots ,\zeta^k (z,w)) = - w , \quad 
s^j (\zeta^1 (z,w), \dots ,\zeta^k (z,w)) = \alpha_j z^j. \]

If we now compute the restricted averaging operator, we have 
\[
\begin{aligned}
 \mathcal{R} (\bar z^r \bar w^s ) &= 
\mathcal{R} (\bar z^r \bar p (\bar z, z)^s) \\
& = \mathcal{R} \left( \sum_{\ell = s}^{ks} A_\ell \bar z^{r + ks - \ell} z^\ell \right) \\
& = \sum_{\ell = s}^{ks} A_\ell \left(\frac{1}{k}\sum_{j=1}^k (\zeta^j (z,w))^{r+ks-\ell} z^\ell   \right).  
 \end{aligned} \]
 Now each of the power sums $\frac{1}{k}\sum_{j=1}^k (\zeta^j (z,w))^{r+ks-\ell}$ is homogeneous (in $z$ and $w$) of 
 degree $r+ks - \ell$, because by the basic theorem 
 on symmetric polynomials, 
 we can rewrite the power sum as a polynomial of the form
 \[ \begin{aligned}
 \frac{1}{k}\sum_{j=1}^k (\zeta^j (z,w))^{r+ks-\ell}
 &= S(s^1 (\zeta(z,w)), \dots , s^k(\zeta(z,w))) \\ &= S(\alpha_1 z, \alpha_2 z^2, \dots , \alpha_{k-1} z^{k-1}, -w)
 \end{aligned}  \]
where $S(x_1, \dots ,x_k)$ is weighted
homogeneous of degree $r+ks - \ell$, when 
$x_j$ has weight $j$.  The claim follows.
\end{proof}

The representation of a real-analytic function as a polynomial 
in $(\bar z, \bar w)$ with holomorphic coefficients is also very 
simple on our models. 

\begin{lemma}
\label{lem:repmodel} Every real-analytic germ $f \in \diffable{\omega} (M_p,0)$
can be written uniquely in the form 
\[ f(z,w,\bar z, \bar w) = \sum_{j=0}^{k-1} f_j (z,w) \bar z^j. \] 
\end{lemma}

\begin{proof}
We use Weierstrass division: First, we divide $f(z, w, \bar z, \bar w)$ 
by $\bar w - \bar p (\bar z, z)$, which yields 
\[ f(z,w, \bar z, \bar w) = \left(\bar w - \bar p (\bar z, z)  \right)q(z,w,\bar z, \bar w) + r(z, w,\bar z), \]
and then the remainder $r$ by $w - p(z,\bar z)$, regarded as a $k$-regular 
function in $\bar z$: 
\[ r(z,w, \bar z) = \left( w -  p (z, \bar z)  \right)\tilde q(z,w,\bar z) +
 \sum_{j=0}^{k-1} f_j (z,w) \bar z^j. \]

For the last part of the statement, 
apply Theorem~\ref{thm:averagingop}.
\end{proof}

Let us remark that the product of two such representations 
can be computed quite efficiently, since one can use the defining equation 
\[ \bar z^k + \alpha_1 \bar z^{k-1} z + \dots + \alpha_{k-1} \bar z z^{k-1} - w = 0 \]
to express $\bar z^j$ for $j\geq k$ recursively. 

An adaptation of that idea is the basis for the proof of the 
following: 

\begin{lemma}
\label{l:generating} The restricted averaging operator 
satisfies 
\[ R(\bar z^a \bar w^b) = \begin{cases}
O(z) & b\geq 1 \text{ or } a \not\cong 0 \mod k, \\
w^{\frac{a}{k}} + O(z) & b=0 \text{ and } a \cong 0 \mod k. 
\end{cases} \]
\end{lemma}

\begin{proof}
Recalling that  $M_p$ is 
given by $w = p(z,\bar z)$ or 
equivalently by $\bar w = \bar p (\bar z, z)$,
where 
\[ p(z,\bar z) = \bar z^k + \sum_{j=1}^{k-1} \alpha_j z^j \bar z^{k-j},  \] 
we have that 
\[ \bar z^a  s^a w  = \sum_j  \alpha_j z^{j} \bar z^{k-j+a}  s^a.  \]
Applying $\restricted$ and summing over $a \in \N$,  we write
\[ R(s) = \sum_{a=0}^\infty \restricted (\bar z^a ) s^a , \quad R_c (s) = \sum_{\substack{a\\ a < c} } \restricted (\bar z^a ) s^a ,  \]
 and obtain
\[  R(s) w = \sum_{j=0}^{k-1} \alpha_j z^j 
\frac{R(s) - R_{k-j} (s)}{s^{k-j}}    \]
so that  
\[ R(s) = \frac{\sum_j \alpha_j z^j s^{j-k}  R_{k-j} (s)}{p(z,s^{-1}) - w} = \frac{1 + O(z)}{1 + O(z) -s^k w}. \]
This proves the claim for $b = 0$. 
For $b\geq 1$, we see from the defining equation of $M_p$ that $\bar w = O(z)$, so that $R(\bar z^a \bar w^b) = O(z)$ also. 
\end{proof}

\label{sub:the_quadric}
\begin{example} Let us discuss the standard quadric in our context. Its defining 
equation will be written as 
\begin{equation}
	\label{e:standardquad} w = \bar z^2 + 2 \mu z \bar z, \quad 0 \leq \mu < \infty
\end{equation}
(we use $\alpha_1 = 2 \mu$ because it simplifies some of the formulas which 
follow). The associated $\zeta^1$, $\zeta^2$ are given by
\[ \zeta^1 (z,w) = - \mu z + \sqrt{w + \mu^2 z^2} , \qquad \zeta^2 (z,w) = - \mu z - \sqrt{w + \mu^2 z^2}.\]
The restricted averaging operator is given by 
\[ \restricted \bar z = - \mu z, \quad \restricted \bar w = (1- 2 \mu^2) z^2, \quad   \]
In the case of the quadric, we can actually also give formulas for $\restricted$ based on 
\[ \begin{aligned}
\restricted \bar z^j &= \frac{1}{2} \left( \left( - \mu z + \sqrt{w + \mu^2 z^2} \right)^j + \left( - \mu z - \sqrt{w + \mu^2 z^2} \right)^j    \right) \\
&= (-1)^j \sum_{p=0}^{\lfloor{j/2}\rfloor} \binom{j}{2p} (\mu z)^{j-2p} (w + \mu^2 z^2)^{p}\\
&=   (-1)^j \sum_{p=0}^{\lfloor{j/2}\rfloor} \sum_{q=0}^{p}\binom{j}{2p} \binom{p}{q} (\mu z)^{j-2p+2q} w^{p-q}; 
\end{aligned}
\]
the formulas for
$\restricted \bar z^j \bar w^k$ can 
be worked out similarly. We also note that the generating function $R(s)$ from 
above is given explicitly by 
\[ R(s) = \frac{1 - \mu z +  2 \mu z s}{1 + 2 \mu z s - w}.\]
\end{example}

The real-valued formal holomorphic maps on a model manifold can be 
explicitly described as follows. 

\begin{thm}
 Assume that $f \in \fps{z,w}$ is real-valued (but not constant) on $M_p$, and that 
 $p(z,\bar z) \neq \bar z^k $. Then there exists a unique $\vartheta \in \R$
 such that  $e^{i \vartheta} p(z,\bar z) + e^{-i \vartheta} z^k = e^{- i \vartheta} \bar p (\bar z, z) + e^{i \vartheta }\bar z^k$ and  
 \[ f(z,w) = \sum_{j=0}^\infty f_j (e^{i \vartheta} w + e^{-i \vartheta} z^k)^j, \quad f_j \in \R.\] 
 \end{thm}

  \begin{proof} We first check that 
 it is enough to determine the weighted homogeneous polynomials $P(z,w)$ 
 with the property that $ \bar P (\bar z, \bar w) = P(z,w)$ on $M$, and 
 to show that these need to be of the form $P(z,w) = f_j (w+z^k)^j$. Indeed,
 if we know the statement for the polynomials, then given any $f \in \fps{z,w}$, 
 we decompose $f = \sum_{j=j_0}^\infty P_j$ with $P_{j_0} \neq 0$. 
 $f$ being real-valued implies that $\restricted (\bar f^a) = f^a $ for $a=1,\dots,k$, 
 so that   $\restricted (\bar P_{j_0}^a) = P_{j_0}^a $ for $a=1,\dots,k$ and thus $P_{j_0}$ 
 is real valued on $M$. The series $\tilde f = f - P_{j_0}$ is therefore of 
 higher vanishing order than $f$ and still real valued. Inductively, we 
 can rewrite $f$ as a sum of real-valued homogeneous polynomials.    

By homogeneity reasons, for $1\leq j < k$, every homogeneous polynomial of degree $j$ is 
necessarily a polynomial of $z$ and therefore never real-valued. 

 For a homogeneous polynomial of degree $k$,
 we have 
 $P_k (z,w) = a z^k + b w$, which restricted to $M_p$ becomes 
 \[ P_k (z, p(z)) = a z^k + b \bar z^k + \sum_{j=1} b \alpha_j z^j \bar z^{k-j}.  \]
 Hence $P_k$ is real-valued on $M_p$ if and only 
 if $a=\bar b$ and $ \bar b \bar \alpha_j = b \alpha_{k-j} $ for $j=1,\dots, k-1$. By assumption, at least one of the $\alpha_j$ is not zero, so that we can find the $\vartheta$ in the polar decomposition $b = f_k e^{- i \vartheta}$ from $p$ alone. Hence, we 
 have $P_k = f_k (e^{i \vartheta} w + e^{-i \vartheta} z^k)  $ as claimed. 

We now proceed by induction on the degree $j$ of $P$, and assume that 
we have proved our claim for all weighted homogeneous polynomials of 
lesser degree. After dividing by $(e^{i \vartheta} w+e^{-i \vartheta}z^k)$, we have that 
\[ P(z,w) = (e^{i \vartheta} w+e^{-i \vartheta}z^k) q(z,w) + r(z), \]
with $q$ weighted homogeneous of degree $j-k$. Since $P = \bar P$ on $M$, we 
have that 
\[ (e^{i \vartheta} w+e^{-i \vartheta}z^k) (q(z,w) - \overline{q(z,w)}) = \overline{r(z)} - r(z) \]
on $M$. Now both the left and the right hand side of this equation 
are imaginary-valued on $M$. One checks as before that this means that $r = 0$. 
It follows that $q (z,w) = \overline{q(z,w)}$, and so the induction hypothesis 
applies to $q$; if $q$ is of weighted degree not divisible by $k$, 
this means that $q=0$, while 
if $q$ is of  weighted degree divisible by $k$, $q$ is a real multiple of $(e^{i \vartheta} w+e^{-i \vartheta}z^k)^j$. 
\end{proof}

\bibliographystyle{abbrv}
\bibliography{bibfile}

\end{document}